\documentclass[12pt]{amsart}

\usepackage{verbatim, amssymb, enumitem}

\usepackage[breaklinks=true,colorlinks=true,linkcolor=blue,citecolor=red,urlcolor=blue,psdextra,pdfencoding=auto]{hyperref}

\newcommand \la{\lambda}
\newcommand \ve{\varepsilon}

\newcommand \br{\mathbb{R}}
\newcommand \bc{\mathbb{C}}
\newcommand \bh{\mathbb{H}}

\newcommand \Ker{\operatorname{Ker}}

\renewcommand \Im{\operatorname{Im}}

\newcommand \Span{\operatorname{Span}}
\newcommand \Tr{\operatorname{Tr}}

\newcommand \SO{\mathrm{SO}}
\newcommand \Or{\mathrm{O}}

\newcommand \SP{\mathrm{Sp}}
\newcommand \Spin{\mathrm{Spin}}

\newcommand \sC{\mathsf{C}}

\newcommand \sN{\mathsf{N}}

\newcommand \cC{\mathcal{C}}

\newcommand\ag{\mathfrak a}
\newcommand\bg{\mathfrak{b}}

\newcommand\z{\mathfrak z}

\newcommand \so{\mathfrak{so}}
\newcommand \spg{\mathfrak{sp}}

\newcommand \su{\mathfrak{su}}

\newcommand \n{\mathfrak{n}}

\newcommand \Ng{\mathfrak{N}}
\newcommand \Pg{\mathfrak{P}}

\newcommand \ad{\operatorname{ad}}
\newcommand \Ad{\operatorname{Ad}}
\newcommand \diag{\operatorname{diag}}

\newcommand \Id{\operatorname{Id}}

\newcommand \<{\langle}
\renewcommand \>{\rangle}
\newcommand \ip{{\<\cdot,\cdot\>}}

\newcommand \ir{\mathrm{i}}
\newcommand \jr{\mathrm{j}}
\newcommand \kr{\mathrm{k}}

\newtheorem{theorema}{Theorem}

\newtheorem*{theorem*}{Theorem}

\newtheorem*{corollary*}{Corollary}
\newtheorem*{conj*}{Conjecture}
\newtheorem{lemma}{Lemma}

\newtheorem*{prop*}{Proposition}

\theoremstyle{definition}

\newtheorem*{definition*}{Definition}

\theoremstyle{remark}

\newtheorem*{notation*}{Notation}
\newtheorem*{algorithm*}{Algorithm}
\newtheorem*{example*}{Example}

\begin{document}

\title{Non-singular weakly symmetric nilmanifolds}

\author{Y.~Nikolayevsky}
\address{Department of Mathematical and Physical Sciences, La Trobe University, Melbourne, Australia 3086}
\email{y.nikolayevsky@latrobe.edu.au}
\thanks{The first named author was partially supported by ARC Discovery Grant DP210100951, and the second named author by an RFG grant from the University of Pennsylvania.}

\author{W.~Ziller}
\address{Department of Mathematics, University of Pennsylvania, Philadelphia, PA 19104, USA} 
\email{wziller@math.upenn.edu}



\dedicatory{The authors dedicate this paper to the memory of Joe Wolf. Among his many remarkable contributions to geometry and Lie theory, are deep results in the theory of weakly symmetric and geodesic orbit spaces. This research is a small contribution to the topic.}

\begin{abstract}
A Riemannian manifold $M$ is called weakly symmetric if any two points in $M$ can be interchanged by an isometry. The compact ones have been well understood, and the main remaining case is that of 2-step nilpotent Lie groups. We give a complete classification of simply connected non-singular weakly symmetric nilmanifolds. Besides previously known examples, there are new families  with 3-dimensional center, and a one-parameter family of dimensions 14. The classification is based on the authors classification of non-singular 2-step nilpotent Lie groups for which every geodesic is the image of a one parameter group of isometries.
\end{abstract}

\maketitle

\section{Introduction}
\label{s:intro}

A complete Riemannian manifold $M$ is called \emph{weakly symmetric} if any two points in $M$ can be interchanged by an isometry. It is called \emph{geodesic orbit}, GO for short, if every geodesic is the image of a one-parameter group of isometries.  Both concepts have been studied extensively. In \cite{BKV} it was shown that weakly symmetric spaces are GO. But the converse does not hold. Clearly, both conditions imply that $M$ is homogeneous, and so $M=G/H$ where $G$ is the full isometry group.

Weakly symmetric spaces fall into several very different categories. When the group $G$ is reductive (e.g., when $M$ is compact and simply connected), such spaces have been classified in \cite{Vi,Ya1,Wo2}. Another category is the one of weakly symmetric nilmanifolds. It is  well known that for a nilpotent Lie group both the GO and the weakly symmetric condition imply that it is either abelian or 2-step nilpotent, see \cite[Theorem~2.2]{Gor}. In this case, one finds a partial classification in \cite{Ya2, Wo2}. This  classification is based on the fact that weakly symmetric spaces are commutative \cite{Sel}, and the partial classification of the former is obtained from the classification of the latter.

\smallskip

To state our results, we recall the following standard construction. Let $(M,g)$ be a metric $2$-step nilpotent Lie group and $(\n, \ip)$ the corresponding metric Lie algebra. Let $\z = [\n,\n]$,  $\ag=\n^\perp$. Here we can assume that $\z$ is the center of $\n$ since otherwise $\n$ is the product of $\n'\subset\n$ with an abelian Lie algebra. The Lie brackets of $\n$ are encoded in the endomorphisms
\begin{equation}\label{jz}
	J_Z \in \so(\ag) \text{ with } \<J_ZX,Y\>=\<Z,[X,Y]\>, \text{ for all } X,Y \in \ag \text{ and } Z\in \z
\end{equation}
which defines a linear subspace:
\begin{equation*}
	V=\Span\{  J_Z\mid Z\in \z\} \subset \so(\ag).
\end{equation*}

Moreover, $V$ inherits an inner product from the push-forward of the inner product on $\z\subset\n$. Conversely, we can start with a Euclidean vector space $\ag$, a  linear subspace $V\subset \so(\ag)$ and an inner product $\ip$ on $V$. This defines a metric nilpotent Lie algebra  $\n$ by setting $\n=V\oplus\ag$ with  $V=[\n,\n]$  and Lie brackets $[\ag,\ag]\subset V$ given by \eqref{jz}. Furthermore, it inherits an inner product  by declaring the decomposition to be orthogonal and using the given inner products on $V$ and $\ag$.  There is a one-to-one correspondence between pairs $(V, \ip)$ with $V\subset \so(\ag)$ and simply connected metric $2$-step nilpotent Lie groups. We call $(V, \ip)$ a \emph{WS-pair} (resp. GO-pair) if the corresponding simply connected nilpotent metric Lie group is weakly symmetric (resp. geodesic orbit). We call an inner product $\ip$ on a subspace $V \subset \so(\ag)$ \emph{admissible} if $(V, \ip)$ is a WS-pair, and \emph{standard} if $\ip$ is the restriction of a bi-invariant inner product on $\so(\ag)$. If $(V, \ip)$ is a WS-pair, then a standard inner product is automatically admissible.

The case where $V$ is a subalgebra of $\so(\ag)$, allowing any bi-invariant metric on $V$, corresponds to the case where the nilpotent Lie group is naturally reductive, see \cite[Theorem~2.5]{Gor}. These are all GO, but WS ones constitute a small subset of them \cite[Theorem~5]{L3}. A second class of special examples are 2-step nilpotent Lie algebras constructed from Clifford modules, so called (generalized) H-type algebras. We call such examples to be of \emph{Clifford type}, allowing a non-standard admissible inner product on $V$. In \cite[Example~6]{Z}, \cite[Theorem~1]{BRV} one finds a classification of WS-pairs of Clifford type with a standard inner product, and in \cite[Corollary~5.11]{L2} a classification with any admissible inner product. There two families, and one  pair in dimension $31$, which are GO but not WS.

A $2$-step nilpotent Lie algebra is called \emph{non-singular} if all nonzero elements of $V$ are invertible. All $2$-step nilpotent Lie algebras of Clifford type have this property. But the class of non-singular Lie algebras is surprisingly large and is not classified, unless $\dim V=1$ or $2$. For the current state of knowledge on non-singular $2$-step nilpotent Lie algebras, we refer the reader to~\cite{E1, LO} and the bibliographies therein. Our goal is to classify  2-step nilpotent non-singular WS-pairs. In \cite{NZ} the authors classified all non-singular GO-pairs. Since weakly symmetric pairs are GO, we can use this classification in order to determine which ones are weakly symmetric.

\smallskip

In the following, $\bh$ denotes the algebra of quaternions with the standard inner product, $L_q$ and $R_q$ the left and the right multiplications by $q \in \bh$, and $\ir,\jr,\kr$ the standard basis of the space of imaginary quaternions $\Im\bh$. Furthermore, $J_1,\dots , J_7$ are seven anticommuting complex structures on $\br^8$.

\begin{theorem*} 
  Let $(V, \ip)$ be a non-singular WS-pair. Then it is either of Clifford type, or belongs to one of  the following:
\begin{enumerate}[label=\emph{(\alph*)},ref=\alph*]
    \item \label{it:ws1} 
    $\dim V =1$ and $V=\br J$ with $J$ a non-singular skew-symmetric matrix.

    \item \label{it:ws2}  
    $\dim V =2$. We  identify $\ag$ with $\bh^p$ and let $V=\Span(J_1, J_2)$, with
     \begin{equation*}
    	 J_1=\diag(L_{a_1}, L_{a_2}, \dots, L_{a_p}),\quad  J_2=\diag(L_{b_1}, L_{b_2}, \dots, L_{b_p}),
    \end{equation*}
    where, for each $s=1\dots p$, the quaternions $a_s, b_s\in\Im\bh$ are linearly independent. The inner product $\ip$ on $V$ is arbitrary.

    \item \label{it:ws3}
    $\dim V = 3$.   We identify $\ag$ with $ \bh^p$ and let $V=\Span(J_1, J_2, J_3)$ with
    \begin{equation*}
    \hspace{30pt}	J_1=\diag(\la_1 L_{\ir}, \dots, \la_p L_{\ir}), \ J_2=\diag(\mu_1 L_{\jr}, \dots, \mu_p L_{\jr}), \ J_3=\diag(\mu_1 L_{\kr}, \dots, \mu_p L_{\kr}),
    \end{equation*}
    where  $\la_s, \mu_s\in \br$ are non-zero. The inner product on $V$ is such that the element $J_1, J_2$ and $J_3$ are mutually orthogonal and $\|J_2\|=\|J_3\|$.

    \item \label{it:ws6} 
    $\dim V = 6$. Then $\dim \ag=8$ and $V=\Span(J_1, J_2, J_3, J_4, J_5,J')$, where $J'=J_7 \cos \theta + J_6J_7 \sin \theta, \; \theta \in (0, \pi/2)$. The inner product on $V$ is such that its restriction to $\Span(J_1, J_2, J_3, J_4, J_5)$ is standard and $J' \perp J_i, \, i =1, \dots, 5$.
\end{enumerate}
\end{theorem*}

Notice that in the non-Clifford case a non-singular nilmanifold with $\dim \z \in \{1, 2, 6\}$ is WS if and only if it is GO, see Theorem~\ref{tha:class} below. But if $\dim V=3$ or $7$, there are non-isomorphic families which are GO, but neither WS nor of Clifford type. We remark that if we allow some of $a_s, b_s$ to be linearly dependent in case~\eqref{it:ws2}, or some of $\la_s, \mu_s$ to be zero in case~\eqref{it:ws3}, the resulting nilpotent Lie group is still WS, but is no longer non-singular. Notice also that in case~\eqref{it:ws6}, different values of $\theta$ give non-isomorphic Lie algebras since the maxima of the maximal eigenvalues of the matrices $J^2, \, J \in V$, with $\Tr (J^2)=-1$, are distinct~\cite[Section~3]{NZ}. We can view all of the nilpotent Lie algebras in the above theorem as deformations of the corresponding ones of Clifford type by changing the structure constants.

In Section~\ref{s:discuss} we discuss the relation of the above classification with previously known results.

\section{Preliminaries}
\label{s:prel}

Given a subspace $V \subset \so(\ag)$ with an inner product $\ip$, denote by $\sN(V)$ its \emph{orthogonal normalizer}:
\begin{equation*}
	\sN(V) = \{N \in \Or(\ag) \, | \, NVN^{-1} \subset V, \text{ and } (\Ad_N)_{|V}\in \Or(V)\}.
\end{equation*}
and by $\sC(V)$ the centralizer:
\begin{equation*}
	\sC(V) = \{N \in \Or(\ag) \, | \, JN=NJ  \text{ for all  } J\in V\}.
\end{equation*}
Clearly, $\sC(V)$ is a normal subgroup of $\sN(V)$, and hence on the Lie algebra level we have a decomposition into ideals:
\begin{equation*}
\Ng(V)=\cC(V)\oplus \Pg(V)
\end{equation*}
We call $\Pg(V)$ the \emph{pure normalizer} which clearly acts effectively  on $V$.

\smallskip

The isometry group of a left invariant metric on $\n=\z\oplus\ag$ is equal to
\begin{equation*}
K=\{  (A,B) \in O(\z)\times O(\ag)\mid BJ_ZB^{-1}=J_{AZ}  \text{ with }J_Z\in V\},
\end{equation*}
see~\cite[Theorem~4.2]{Wo1}.
Thus the condition for the nilmanifold to be weakly symmetric is equivalent to the following: for any $J \in V$ and any $X \in \ag$, there exists an element $N \in \sN(V)$ such that
\begin{equation}\label{eq:wscond}
	NX = -X \qquad \text{and } NJ=-JN.
\end{equation}
For simplicity, we will also write the latter equation as $N\cdot J = -J$, where $N\cdot J=N^tJN.$

We start by recalling the classification of non-singular GO 2-step nilpotent Lie algebras.

\begin{theorema} [{\cite[Theorem~1]{NZ}}] \label{tha:class}
  Let $(V, \ip)$ be a $2$-step nilpotent non-singular GO-pair. Then  it is either of Clifford type, or belongs to one of  the following:
  \begin{enumerate}[label=\emph{(\alph*)},ref=\alph*]
    \item \label{it:th1} $\dim V =1$. Then $V$ is spanned by a non-singular skew-symmetric matrix.

    \item \label{it:th2} $\dim V =2$. We  identify $\ag$ with $\bh^p$ and then $V=\Span(J_1, J_2)$ with
    \begin{equation*}
    	J_1=\diag(L_{a_1}, L_{a_2}, \dots, L_{a_p}),\quad  J_2=\diag(L_{b_1}, L_{b_2}, \dots, L_{b_p}),
    \end{equation*}
    where $a_s, b_s \in \Im \bh$ are linearly independent.

    \item \label{it:th3}
    $\dim V =3$. Two cases are possible:
    \begin{enumerate}[label=\emph{(\roman*)},ref=\roman*]
      \item \label{it:th3cen}
      Identify $\ag$ with $\bh^p$. Then  $V=\Span(J_1, J_2, J_3)$ with
      \begin{gather*}
     	\hspace{30pt}  J_1=\diag(L_{a_1}, L_{a_2}, \dots, L_{a_p}),\quad  J_2=\diag(L_{b_1}, L_{b_2}, \dots, L_{b_p}), \\
        \hspace{30pt} J_3=\diag(L_{c_1}, L_{c_2}, \dots, L_{c_p}),
      \end{gather*}
      where, for every $s = 1, \dots p$, the quaternions  $\{a_s,b_s,c_s\}$ are a basis of $\Im\bh$.

      \item \label{it:th3rep}
      Let $\ag= \bh^p \oplus W, \, p \ge 0$, and $\rho\colon\so(3)\to \so(W)$ be a quaternionic representation on $W=\br^{4q}, \, q>1$, with no $4$-dimensional sub-representations. Then $V=\Span(J_1, J_2, J_3)$ with
      \begin{gather*}
      	\hspace{30pt}  J_1=\diag(\la_1 L_{\mathrm{i}}, \dots, \la_p L_{\mathrm{i}}, \rho(\mathrm{i})),\quad  J_2 = \diag(\la_1 L_{\mathrm{j}}, \dots, \la_p L_{\mathrm{j}},\rho(\mathrm{j}) ), \\
      \hspace{30pt}  J_3=\diag(\la_1 L_{\mathrm{k}},  \dots, \la_p L_{\mathrm{k}},\rho(\mathrm{k})),
      \end{gather*}
     where $\la_s \ne 0$ and where we identify $\so(3)$ with $\Im\bh$.
    \end{enumerate}

    \item \label{it:th6}
    $\dim V =6$. Then $\dim \ag=8$ and $V=\Span(J_1, J_2, J_3, J_4, J_5,J')$, where $J'=J_7 \cos \theta + J_6J_7 \sin \theta, \; \theta \in (0, \pi/2)$.

    \item \label{it:th7}
    $\dim V =7$. Then
   $\dim \ag=8$ and $V=\Span(J_1, J_2, J_3, J_4, J_5,J_6,J')$, where $J'=J_7 \cos \theta + J_6J_7 \sin \theta, \; \theta \in (0, \pi/2)$.
  \end{enumerate}
\end{theorema}

Admissible inner products for the GO-pairs in Theorem~\ref{tha:class} are as follows:
\begin{itemize}
	\item If $\dim V \in \{1,2\}$ or $\dim V = 3$ in case~\eqref{it:th3}\eqref{it:th3cen} (these are subspaces of so-called \emph{centralizer type}), then any inner product on $V$ is admissible.
	\item If $\dim V=3$ in case~\eqref{it:th3}\eqref{it:th3rep}, the inner product is standard.
	\item If $\dim V =6, 7$, then the restriction of the inner product to $\Span(J_1, J_2, J_3, J_4, J_5)$ is standard, and $J' \perp J_i$, respectively  $J_6,J' \perp J_i$, for $i=1, \dots, 5$.
\end{itemize}

We also state the classification of GO-pairs of Clifford type. Let $\ip$ be an inner product on $V$ and let $S$ be the symmetric operator defined by $\<J, J\> = (SJ,J)$ for $J \in V$, where $(\cdot, \cdot)$ is a standard inner product. Let $\la_1, \dots, \la_q$ be the eigenvalues of $S$ with multiplicities $m_1, \dots, m_q$, respectively. We say that the inner product $\ip$ has \emph{eigenvalue type} $(m_1, \dots, m_q)$; so a standard inner product has eigenvalue type $(\dim V)$.

 \begin{theorema} [{\cite{Z}, \cite{BRV}, \cite[Corollary~5.11]{L2}}] \label{tha:Cliff}
  A WS-pair $(V, \ip)$ is of Clifford type precisely in the following cases.
 	\begin{enumerate}[label=\emph{(\alph*)},ref=\alph*]
 		\item \label{it:Cliff12}
 		$\dim V \in \{1,2\}$ and $\ip$ is arbitrary.
 		
 		\item \label{it:Cliff3}
 		$\dim V = 3$ and the eigenvalue type of $\ip$ is either $(3)$ or $(2,1)$.
 		
 		\item \label{it:Cliff5}
 		$\dim V = 5, \; \dim \ag = 8$ and $\ip$ is standard.
 		
 		\item \label{it:Cliff6}
 		$\dim V = 6, \; \dim \ag = 8$ and the eigenvalue type of $\ip$ is either $(6)$ or $(5,1)$.
 		
 		\item \label{it:Cliff7}
 		$\dim V = 7, \; \dim \ag = 8$ and the eigenvalue type of $\ip$ is either $(7)$ or $(5,2)$.
 		
 		\item \label{it:Cliff73}
 		$\dim V = 7, \; \dim \ag = 16$, the inner product $\ip$ is standard, and $V$ is spanned by $7$ pairwise anticommuting complex structures $K_i$ such that $K_1 \dots K_7 = \pm \Id_\ag$.
 	\end{enumerate}
 \end{theorema}

Notice that by Theorem~\ref{tha:class}, the nilpotent Lie groups in~\eqref{it:Cliff5} and~\eqref{it:Cliff73} admit no deformations of the structure constants (which of course preserves the condition of being non-singular) which are also GO (and a fortiori, WS), whereas the remaining ones do.

\section{Examples are weakly symmetric}
\label{s:nsws}

\medskip

We start by discussing the examples in our main theorem.

\medskip

Suppose $\dim V = 1$. We can identify $\ag$ with $\bc^p$ such that $J$ acts by multiplication by an imaginary number in each coordinate. If $P(X)=\overline{X}$, then  conjugation by $P$ takes $J$ to $-J$. We can now choose an element $Q$ in maximal torus    $\diag(e^{i\theta_1},\cdots,e^{i\theta_p} )$ which lies in the centralizer, with $Q(PX)= -X$ for all $X\in \ag$. Thus $QP$  satisfy property~\eqref{eq:wscond}. The inner product $\ip$ on $V$ is clearly arbitrary.

\medskip

Suppose $\dim V = 2$ with $ J_1=\diag(L_{a_1}, L_{a_2}, \dots,  L_{a_p}),\quad  J_2=\diag(L_{b_1}, L_{b_2}, \dots, L_{b_p})$ with $a_s,b_s$ imaginary quaternions. For each $s$, there exists a unit quaternion $q_s$ which is orthogonal to $a_s$ and $b_s$, and which anti-commutes with both $a_s$ and $b_s$. Hence  $P = \diag(L_{q_1}, \dots, L_{q_p})$ induces the action of $-\Id$ on $V$.  We can furthermore chose unit quaternions $r_s$ such that $R_{r_s}(L_{q_s}  X_s) = -X_s$, where $X_s \in \bh$ is the corresponding component of $X \in \ag$. Defining $Q = \diag(R_{r_1}, \dots, R_{r_p})$ it follows that $Q(PX)=-X$. Since $Q$ also commutes with $P$, this implies  that  $N = QP\in \Or(\ag)$ satisfies~\eqref{eq:wscond}.  The inner product $\ip$ on $V$ is arbitrary.

\medskip

Suppose $\dim V = 3$.  We have $V = \Span(J_1,J_2,J_3)$ with
\begin{equation*}
	J_1=\diag(\la_1 L_{\ir}, \dots, \la_p L_{\ir}), \quad J_2=\diag(\mu_1 L_{\jr}, \dots, \mu_p L_{\jr}), \quad J_3=\diag(\mu_1 L_{\kr}, \dots, \mu_p L_{\kr}),
\end{equation*}
where  $\la_s, \mu_s\in \br$ are non-zero.
Given an arbitrary $J = a J_1 + b J_2 + c J_3 \in V$ with $a,b,c \in \br$ and $X \in \ag$, let $r$ be a unit imaginary quaternion orthogonal to both $\ir$ and $b \jr + c \kr$. Then the element $P=\diag(L_r,L_r,\dots, L_r) \in \Or(\ag)$ anti-commutes with $J$. Moreover, $P$ lies in $ \sN(V)$ since  conjugation by $P$ sends $J_1$ to $-J_1$ and acts as an orthogonal transformation with determinant $-1$ on $\Span(J_2, J_3)$. We can now chose an element in the centralizer of $V$ of the form $Q = \diag(R_{r_1}, \dots, R_{r_p})$, for some unit quaternions $r_i$, such that $Q(PX)=-X$. Thus $N=QP$  satisfies the equations $N \cdot J = -J$ and $NX=-X$, as required.

\smallskip Finally, notice that $ \sN(V)$ acts by isometries in an inner product  on $V$ if and only if
 $J_1, J_2$ and $J_3$ are mutually orthogonal and $\|J_2\|=\|J_3\|$.

\medskip

Suppose $\dim V = 6$. Here we have  $V=\Span(J_1, J_2,$ $J_3, J_4, J_5,J')$ with $J'=J_7 \cos \theta + J_6J_7 \sin \theta, \; \theta \in (0, \pi/2)$,  where $J_1, \dots, J_7$ are pairwise anticommuting complex structures on $\br^8$.
By \cite[Lemma 7]{NZ} the subgroup $\SP(2)\subset \SO(8)$, with its standard embedding,  lies in $\sN(V)$ and acts on $\Span(J_1, J_2,$ $J_3, J_4, J_5)\subset V$ as $\SO(5)$  under the two-fold cover $\SP(2)\to \SO(5)$.
  Furthermore, $J_6$ clearly also lies in  $\sN(V)$ and acts, via conjugation, as $-\Id$ on $V$. For a generic $J\in V$ the subgroup of $\SO(5)$ that commutes with $J$ is equal to $\SO(4)\subset \SO(5)$. Its lift under the two-fold cover is   $\SP(1)\SP(1)\subset \SP(2)$  with its block embedding.  Since the action of $\SP(2)$ on $\ag=\bh^2$ is standard, $\SP(1)\SP(1)\subset \SP(2)$    acts  on each component of $\bh^2=\bh\oplus\bh$ separately. Thus for each $X\in\ag$, there exists a $Q\in \SP(1)\SP(1)$ with $Q(J_6X)=-X$ and hence $N=QJ_6$  satisfies~\eqref{eq:wscond}.
\smallskip

The  subgroup of $\sN(V)$ generated by $\SP(2)$ and $J_6$ acts by isometries in  an inner product  on $V$ if and only if its  restriction  to $\Span(J_1, J_2, J_3, J_4, J_5)$ is standard and $J' \perp J_i$, for $i=1, \dots, 5$.

\smallskip

Notice that in all of these examples, the eigenvalue type of $\ip$ on $V$ is the same as those of Clifford type, and we simply deform the structure constants of $\n$.

\section{Proof of the Theorem}
\label{s:proof}

We now need to show that the remaining GO-pairs in Theorem~\ref{tha:class} are not WS-pairs.

\smallskip

First consider the GO-pairs in case~\eqref{it:th3}\eqref{it:th3rep} of Theorem~\ref{tha:class}. It is sufficient to show that $\sN(V) \subset \Or(\bh^p) \times \Or(W) \subset \Or(\ag)$. Indeed, this fact implies that if the condition~\eqref{eq:wscond} is satisfied for $(V, \ip)$, then it is also satisfied with $\ag$ replaced by $W$, and $V$ replaced by its projection $V'$ to $\so(W)$. But such a $V'$ is the image of a skew-symmetric representation of the algebra $\su(2)$ on $W$ which is the direct sum of representations of dimension at least $8$, and then no such space $V'$ can be WS by \cite[Theorem~5 and Theorem~3]{L3}. The proof of the fact that $\sN(V) \subset \Or(\bh^p) \times \Or(W)$ is a Schur-type argument. Let $1 \le r \le p$, and let $W_s \subset W$ be an irreducible $\so(3)$-submodule of dimension $4m \, (\ge 8)$, with $\rho_s$ the corresponding representation. For $N \in \sN(V)$ denote $N_{rs}$ the $4 \times(4m)$ block of the matrix $N$ positioned in the rows corresponding to the $r$-th factor $\bh$ of $\bh^p$ and the columns corresponding to $W_s$. For any $a=1,2,3$, we have $J N = N J_a$, for some $J \in V$, and so for any $q \in \Im \bh$ there exists $q' \in \Im \bh$ such that $\la_r L_{q'} N_{rs} = N_{rs} \rho_s(q)$. Let $W' = \Ker N_{rs} \subset W_s$. Then the subspace $W'$ is nontrivial as $4 m > 4$ and $\rho_s$-invariant. As $\rho_s$ is irreducible, we obtain $W'=W$, and so $N_{rs}=0$. Thus the subspace $W \subset \ag$ is $\sN$-invariant, as required.

\smallskip

Next we consider the GO-pairs in case~\eqref{it:th3}\eqref{it:th3cen} of Theorem~\ref{tha:class}. For $s=1, \dots, p$, denote $\bh_s$ the $s$-th factor of $\bh^p$ and define the (non-singular) maps $A_s: \z \to \Im\bh$ such that $J_Z= \diag(L_{A_1Z}, \dots, L_{A_pZ})$. Then any $N \in \sN(V)$ induces the map $\phi_N: \z \to \z$ such that $N \cdot J_Z=J_{\phi_N Z}$. The map $\phi_N$ is orthogonal relative to both the given inner product on $\z$ and the standard one (inherited from the restriction of the Killing form on $\so(\ag)$ to $V$). Let $N_{rs}$ be the $4 \times 4$ block of the matrix $N$ positioned in the rows corresponding to $\bh_r$ and the columns corresponding to $\bh_s$. For all $Z \in \z$ we have
\begin{equation}\label{eq:Nrs}
  L_{A_rZ} N_{rs} = N_{rs} L_{A_s(\phi_N Z)}.
\end{equation}

We prove the following.

\begin{lemma} \label{l:Nsr}
  In the above notation, for the GO-pairs in case~\eqref{it:th3}\eqref{it:th3cen} of Theorem~\ref{tha:class}, the following holds.
  \begin{enumerate}[label=\emph{(\alph*)},ref=\alph*]
    \item \label{it:Nrsw}
    Suppose that $N_{rs} \ne 0$ for some $s,r=1, \dots, p$. Then there exist $w_{rs}, w'_{rs} \in \bh$ with $\|w_{rs}\|=1,\, w_{rs}' \ne 0$, such that $N_{rs}=L_{w_{rs}} R_{w'_{rs}}$. Moreover, $w_{rs}^{-1}(A_rZ)w_{rs}=A_s(\phi_NZ)$, for all $Z \in \z$, and we have $\det(A_sA_r^{-1}) = \det(\phi_N) \, (\in \{\pm 1\})$.

    \item \label{it:Nrsnond}
    The group $\sN(V)/\sC(V)$ cannot be discreet.
\end{enumerate}
\end{lemma}
\begin{proof}
  To prove assertion~\eqref{it:Nrsw}, for some fixed $s,r=1, \dots, p$, denote $Q=N_{rs} \ne 0$ and define the linear map $\psi:\bh \to \bh$ by $\psi_{|\Im\bh} = (A_s  \circ \phi_N \circ A_r^{-1})_{|\Im\bh}$ on $\Im\bh$ and by $\psi(1)=1$; note that $\psi$ is non-singular. From~\eqref{eq:Nrs} we obtain $L_{u} Q = Q L_{\psi(u)}$, for all $u \in \bh$. If $X \in \Ker Q$, then also $L_{\psi(u)}X \in \Ker Q$ which implies that the subspace $\Ker Q \subset \bh$ is $L_v$-invariant, for all $v \in \bh$. It follows that $Q$ is non-singular, since $Q \ne 0$. Then $QL_{\psi(uv)}=L_{uv}Q=L_uL_vQ=QL_{\psi(u)}L_{\psi(v)} = QL_{\psi(u)\psi(v)}$, for all $u,v \in \bh$, and so $\psi:\bh \to \bh$ is an isomorphism. Then $\psi(u)=w^{-1}uw$ for some unit $w \in \bh$, which gives $L_uQL_w^{-1}=QL_w^{-1}L_u$, for all $u \in \bh$. Thus $QL_w^{-1}$ commutes with all the left multiplications; it follows that it is a right multiplication, so that $Q= L_wR_{w'}$, for some $w' \in \bh$. Then $w^{-1}(A_rZ)w=A_s(\phi_NZ)$, for all $Z \in \z$, as required. In particular, $\det(A_sA_r^{-1}) = \det(\phi_N)$, and $\det(\phi_N) \in \{\pm 1\}$, as $\phi_N$ is orthogonal.

  For assertion~\eqref{it:Nrsnond}, suppose that the group $\sN(V)/\sC(V)$ is discreet. Since both groups are compact, this means that $\sN(V)/\sC(V)$ is finite. As the map $N \mapsto \phi_N$ is well-defined on the cosets from $\sN(V)/\sC(V)$, its image is a finite subset of $\Or(\z)$. By~\eqref{eq:wscond}, any $Z \in \z$ belongs to the $(-1)$-eigenspace of $\phi_N$, for some $N \in \sN(V)$, and so there exists an open subset $U \subset \z$ such that the condition $N \cdot J_Z = -J_Z$ for $Z \in U$ is only satisfied with some $N \in \sN(V)$ such that $\phi_N = -\Id$. From assertion~\eqref{it:Nrsw} it follows that for $s,r=1, \dots, p$, the block $N_{rs}$ of such an $N$ can only be nonzero if $\det(A_sA_r^{-1}) = -1$. Consider the orthogonal decomposition $\ag=\bg_1 \oplus \bg_2 \oplus \bg_3$, where
\begin{equation*}
  \bg_1 = \bigoplus\nolimits_{r: \det(A_1A_r^{-1}) = 1} \bh_r, \quad \bg_2 = \bigoplus\nolimits_{r: \det(A_1A_r^{-1}) = -1} \bh_r, \quad \bg_3 = \bigoplus\nolimits_{r: \det(A_1A_r^{-1}) \notin \{\pm 1\}} \bh_r.
\end{equation*}
  Then $N \bg_1 \subset \bg_2, \, N \bg_2 \subset \bg_1$ and $N \bg_3 \subset \bg_3$. Clearly $\bg_1 \ne 0$ and moreover, as $N$ is orthogonal, we have $\dim \bg_1 = \dim \bg_2$ and
\begin{equation*}
  N = \left(\begin{array}{ccc} 0 & Q_1 & 0 \\ Q_2 & 0 & 0 \\ 0 & 0 & Q_3 \end{array} \right),
\end{equation*}
relative to the decomposition $\ag=\bg_1 \oplus \bg_2 \oplus \bg_3$, where the matrices $Q_1, Q_2$ and $Q_3$ are orthogonal. But then the second equation of~\eqref{eq:wscond} is violated if we take $X$ to be a nonzero vector from $\bg_1$.
\end{proof}

Consider the pure normalizer $\Pg(V) \subset \so(\ag)$. From~\cite[Lemma~2(d)]{NZ} it is isomorphic to a subalgebra of $\so(3)$ and lies in the subspace $\oplus_{s=1}^p \so(\bh_s)$. With our choice of the basis, any element $\diag(R_{u_1}, \dots, R_{u_p})$, where $u_s \in \Im\bh$, lies in the centralizer $\cC(V)$, and so $\Pg(V)$ must be a subalgebra of the algebra $\oplus_{s=1}^p \so(3)$, spanned by the elements $\diag(L_{u_1}, \dots, L_{u_p})$, where $u_s \in \Im\bh$. Moreover, the subalgebra $\Pg(V)$ cannot be trivial, by Lemma~\ref{l:Nsr}\eqref{it:Nrsnond}. Let $P = \diag(L_{u_1}, \dots, L_{u_p}) \in \Pg(V)$ be nonzero. Then $\|u_1\| = \dots = \|u_p\|$ as $V$ must be $\ad_{P}$-invariant. Choosing an appropriate basis in each coordinate, we can assume that $P = \diag(L_{\ir}, \dots, L_{\ir})$. As $V$ is an $\ad_{P}$-module, we obtain $V = \Span(J_1,J_2,J_3)$, where
\begin{equation*}
  J_1=\diag(\la_1 L_{\ir}, \dots, \la_p L_{\ir}), \quad J_2=\diag(\mu_1 L_{\jr}, \dots, \mu_p L_{\jr}), \quad J_3=\diag(\mu_1 L_{\kr}, \dots, \mu_p L_{\kr}),
\end{equation*}
for some nonzero constants $\la_s, \mu_s, \, s = 1, \dots, p$ (by additionally rotating the subspaces $\Span(\jr, \kr) \subset \bh_s$ by a choice of the basis, if necessary). Moreover, the action of $\ad_{P}$ on $V$ must be skew-symmetric, and so an admissible inner product $\ip$ on $V$ must satisfy $J_1 \perp J_2, J_3$ and $\|J_2\|=\|J_3\|$, as in~\eqref{it:ws3} of the main Theorem.

\bigskip

Finally, we consider the case where $\dim V = 7$. By Theorem~\ref{tha:class}\eqref{it:th7}, we have $\dim \ag=8$ and $V=\Span(J_1, J_2,$ $J_3, J_4, J_5, J_6, J')$, where $J'=J_7 \cos \theta + J_6J_7 \sin \theta, \; \theta \in (0, \pi/2)$. We prove that for $J_6 \in V$, there is no $N \in \sN(V)$ such that $N \cdot J_6 = -J_6$, and hence condition~\eqref{eq:wscond} cannot possibly be satisfied. It is sufficient to show this assuming the inner product on $V$ to be standard; we take $\|J\|^2 = -\frac18 \Tr J^2$ for $J \in V$, so that the above basis for $V$ is orthonormal (recall that $\Tr (J_iJ_6J_7) = 0$ for all $i=1, \dots, 5$).

For a vector $u=(u_1, \dots, u_5) \in \br^5$, denote $J_u=\sum_{i=1}^{5} u_iJ_i$. Take an element $J=J_u + a J_6 + bJ' \in V$ with $a, b \in \br$. We have $J^2 = -(\|u\|^2+a^2+b^2) I_8 + 2 b \sin \theta J_uJ_6J_7$. The operator $A=J_uJ_6J_7$ is symmetric, has zero trace, and $A^2=\|u\|^2 I_8$, and so the eigenvalues of $A$ are $\pm \|u\|$, both of multiplicity $4$ (provided $u \ne 0$). Therefore $J^2$ has two eigenvalues, $-(\|u\|^2+a^2+b^2) \pm 2 b \, \|u\| \sin \theta = -\|J\|^2 \pm 2\sin \theta \<J,J'\>(\sum_{i=1}^5\<J,J_i\>^2)^{1/2}$, both of multiplicity $4$ (provided $b \, \|u\| \ne 0$). For $N \in \sN(V)$, the conjugation by $N$ acts orthogonally on $V$ and preserves the eigenvalues of $J^2$. It follows that it preserves the function $\Phi(J) = \Tr ((J^2+\|J\|^2 I_8)^2) = 4 \sin^2 \, \theta \<J,J'\>^2 \sum_{i=1}^5\<J,J_i\>^2$ for $J \in V$. Hence for any $J \in V$ and $N \in \sN(V)$ we have \begin{equation}\label{eq:dimV7}
\<J,J'\>^2\sum\nolimits_{i=1}^5\<J,J_i\>^2=\<J,N \cdot J'\>^2\sum\nolimits_{i=1}^5\<J,N \cdot J_i\>^2.
\end{equation}
Taking $J \perp J'$ in~\eqref{eq:dimV7} we obtain $\<J,N \cdot J'\>^2\sum\nolimits_{i=1}^5\<J,N \cdot J_i\>^2=0$, and so $J \perp N \cdot J'$. It follows that the subspace $\br\,J' \subset V$ is $N \cdot $-invariant, and so $N \cdot J'=\pm J'$. Then the subspace $(J')^\perp = \Span(J_1, J_2, J_3, J_4, J_5, J_6)$ is also $N \cdot $-invariant, and so $\sum_{i=1}^6\<J,J_i\>^2=\sum_{i=1}^6\<J,N \cdot J_i\>^2$ which, together with~\eqref{eq:dimV7} and the fact that $N \cdot J'=\pm J'$ gives $\<J,N \cdot J_6\> = \pm\<J,N \cdot J_6\>$. Thus the subspace $\br\,J_6 \subset V$ is also $N \cdot $-invariant, and $N \cdot J_6=\pm J_6$. Then the subspace $\Span(J_1, J_2, J_3, J_4, J_5) = (\Span(J_6, J')^\perp \cap V$ is also $N \cdot $-invariant. Since for any orthonormal basis $\{v^1, \dots, v^5\} \subset \br^5$ we have $J_{v^1} \dots J_{v^5} = \pm J_1 \dots J_5$ we obtain $N \cdot (J_1 \dots J_5) = \pm J_1 \dots J_5$. But $J_1 \dots J_5J_6J_7=\pm I_8$ which implies $N \cdot (J_6J_7)=\pm J_6J_7$. Since we already know that $N \cdot J_6=\pm J_6$, this gives $N \cdot J_7=\pm J_7$. Hence for any $N \in \sN(V)$ there exist $\ve_6,\ve_7,\ve' \in \{\pm 1\}$ such that $N \cdot J_6=\ve_6 J_6,\, N \cdot J_7=\ve_7 J_7$ and $N \cdot J'=\ve' J'$. These equations give $\ve_7 J_7 \cos \theta + \ve_6\ve_7 J_6J_7 \sin \theta = \ve'(J_7 \cos \theta + \ve_6\ve_7 J_6J_7 \sin \theta)$, and so $\ve'=\ve_7=\ve_6\ve_7$. But then $\ve_6 = 1$, and so for no $N \in \sN(V)$ we can have $N \cdot J_6 = -J_6$, in violation of~\eqref{eq:wscond}.

\smallskip

This completes the proof of our main  Theorem.

\section{Appendix}
\label{s:discuss}
  We discuss the relation of our classification with the known partial classification of WS nilmanifolds given in \cite[Theorem~5.7]{Ya2}, \cite[Theorem~15.4.10]{Wo2} and \cite[Table~3]{Vi}. We will be specifically interested in placing the WS nilmanifold obtained from the WS-pair in part (d) of our main theorem   in this context. The said classification is based on the fact that a WS space is commutative in the sense of~\cite{Sel}.

  It is known that any commutative nilmanifold is at most $2$-step nilpotent, and in the above references, the authors give the classification of all irreducible, maximal commutative nilmanifolds $(N \rtimes H)/H$, where $N$ is a nilpotent group with a left-invariant metric and $H$ is a group of  isometric automorphisms of $\n$. We can also assume that  the center $\z$ of the Lie algebra $\n$ of $N$ and coincides with the derived algebra of $\n$. A nilmanifold $(N \rtimes H)/H$ is called \emph{irreducible}, if  $H$ acts irreducibly on $\ag$. Given a commutative nilmanifold $(N \rtimes H)/H$, let $\z_0 \subset \z$ be an $H$-module, with $Z_0 \subset N$ the corresponding subgroup. Then the space $(N/Z_0 \rtimes H)/H$ is again commutative. This process of passing from $N$ to $N/Z_0$ is called a central reduction. A \emph{maximal} commutative nilmanifold is a one that cannot be obtained by a central reduction from a bigger one. After giving the classification of irreducible, maximal commutative nilmanifolds, the authors prove that all of them, except for one, are weakly symmetric (notice however that almost none of them are non-singular).

  This exceptional space is given by $\ag = \bh^p, \, H = \SP(p)$ and $\z = S^2_0 \bh^p  \oplus \Im \bh$, where $S^2_0 \bh^p$ is the $\SP(p)$-module of $\bh$-hermitian $p \times p$ matrices with trace $0$, and $\Im \bh$ is the trivial $\SP(p)$-module. In our notation, $\Im \bh \subset \so(4p) = \so(\ag)$ is the subalgebra $\spg(1)=\{ R_{\mathrm{i}}, R_{\mathrm{j}}, R_{\mathrm{k}} \}$  of right multiplications on $\bh^p$ and $\SP(p) \subset \sN(V)$ acts by left multiplication.

  The nilmanifold in~\cite[Theorem~5]{L1} is constructed by taking the central reduction by the subspace $S^2_0 \bh^2\subset\z$ and then choosing the inner product on $\Im(\bh)$ whose eigenvalue type is $(1,1,1)$. This prevents the group of all isometric automorphisms to be bigger than $\SP(2)$. The resulting nilmanifold is commutative, but  not WS. However, if one takes the eigenvalue type $(2,1)$ instead, one obtains a WS nilmanifold since $\sN(V)$ also contains the subgroup $\Or(2)$. This example is a special case of nilmanifolds in part~\eqref{it:ws6} of the main Theorem.

  In case $p = 2$, the above maximal nilmanifold admits an alternative description in terms of pairwise commuting complex structures $J_1, \dots, J_7$ on $\ag = \br^8$. Namely, the space $V \subset \so(8)$ is the sum of the $5$-dimensional $\SP(2)$-module $V_1$ spanned by the by $J_1, \dots, J_5$, and the 3-dimensional trivial module $\Span(J_6, J_7, J_6 J_7)$, where $\SP(2) \,(= \Spin(5))$ can be viewed as the connected Lie group with the Lie algebra spanned by the elements $J_i J_j, \, 1 \le i < j \le 5$. The resulting nilmanifold is commutative and not-WS. However some of its central reductions \emph{are} weakly symmetric.

  Consider all central reductions such that the resulting subspace $V \subset \so(8)$ is the given by $V = V_1 \oplus V'$, where $V'$ is a proper subspace of $\Span(J_6, J_7, J_6 J_7)\simeq \spg(1)$. The inner product $\ip$ on $V$ is such that its restriction to $V_1$ is standard and that $V' \perp V_1$. All the resulting nilmanifolds are commutative, but some are weakly-symmetric, and some are not. Let $J'=J_7 \cos \theta + J_6 J_7 \sin \theta$. Up to isometry, all possible cases are as follows:
  \begin{itemize}
    \item
    If $V' = 0$, then the pair $(V,\ip)$ is WS, and is of Clifford type \cite{Z}, \cite{BRV}, \cite[Corollary~5.11]{L2}.

    \item
    If $V' = \br \cdot J'$, the pair $(V,\ip)$ is WS, for any $\theta \in [0, \pi/2]$. If $\theta \in (0, \pi/2)$, this is the WS-pair in part~\eqref{it:ws6} of our main Theorem. The same proof also works for $\theta = 0, \pi/2$ (but in the latter case, the subspace $V$ is not non-singular). If $\theta=0$, the WS-pair of Clifford type \cite[Corollary~5.11]{L2}.

    \item
    If $V' = \Span(J_6, J'), \, \theta \in (0, \pi/2)$, then the pair $(V,\ip)$ is not WS, for any inner product on $V'$, see Section~\ref{s:nsws}.

    \item
    If $V'=\Span(J_6, J_7)$, then the pair $(V,\ip)$ is WS if and only if the restriction of $\ip$ to $V'$ is standard, see~\cite[Corollary~5.11]{L2}. The resulting WS-pair is of Clifford type. Note that in this case, $\sN(V) \supset \SP(2) \times \mathrm{U}(1)$.

    \item
    If $V' = \Span(J_6, J_6J_7)$, then the pair $(V,\ip)$ is WS, for any inner product on $V'$, but $V$ is not non-singular. In this case, the group $\sN(V)$ contains the element $J_7$, the conjugation by which acts as $-\Id$ on $V$. Then an argument similar to the one in the case $\dim V = 6$ in Section~\ref{s:nsws} shows the pair is WS.
  \end{itemize}


\begin{thebibliography}{BKV}

\bibitem[BKV]{BKV}
J.~Berndt, O.~Kowalski and L.~Vanhecke, \emph{Geodesics in weakly symmetric spaces}, Ann. Global Anal. Geom.  \textbf{15} (1997), 153--156.

\bibitem[BRV]{BRV}
J.~Berndt, F.~Ricci, L.~Vanhecke, \emph{Weakly symmetric groups of Heisenberg type}, Diff. Geom. Appl. \textbf{8} (1998), 275--284.

\bibitem[E1]{E1}
P.~Eberlein, \emph{Geometry of 2-step nilpotent Lie groups with a left invariant metric II}, Trans. Amer. Math. Soc. \textbf{343} (1994), 805--828.

\bibitem[Gor]{Gor}
C.~Gordon,
\emph{Homogeneous manifolds whose geodesics are orbits}, in: Topics in Geometry in Memory of Joseph D'Atri, Birkh\"{a}user, Basel, 1996, 155--174.

\bibitem[L1]{L1}
J.~Lauret, \emph{Commutative spaces which are not weakly symmetric}, Bull. London Math. Soc. \textbf{30} (1998), 29--37.

\bibitem[L2]{L2}
J.~Lauret, \emph{Modified $H$-type groups and symmetric-like Riemannian spaces}, Diff. Geom. Appl. \textbf{10} (1999), 121--143.

\bibitem[L3]{L3}
J.~Lauret, \emph{Weak symmetry in naturally reductive homogeneous nilmanifolds}, Rocky Mountain J. Math. \textbf{34} (2004), 215--224.

\bibitem[LO]{LO}
J.~Lauret, D.~Oscari, \emph{On non-singular 2-step nilpotent Lie algebras}, Math. Res. Lett. \textbf{21} (2014), 553--583.

\bibitem[NZ]{NZ}
Y.~Nikolayevsky, W.~Ziller, \emph{Non-singular geodesic orbit nilmanifolds}. In preparation.

\bibitem[Sel]{Sel}
A.~Selberg, \emph{Harmonic analysis and discontinues groups in weakly symmetric spaces with applications to Dirichlet series}, J. Indian Math. Soc. (N.S.) \textbf{20} (1956), 47--87.

\bibitem[Vi]{Vi}
E.~B.~Vinberg, \emph{Commutative homogeneous spaces and co-isotropic symplectic actions},
Russian Math. Surveys \textbf{56} (2001), 1-60.

\bibitem[Wo1]{Wo1}
J.~A.~Wolf, \emph{On Locally Symmetric Spaces of Non-negative Curvature and Certain Other Locally Homogeneous Spaces}, Comm. Math. Helv. \textbf{37} (1963), 265--295.

\bibitem[Wo2]{Wo2}
J.~A.~Wolf, \emph{Harmonic analysis on commutative spaces}.  Mathematical Surveys and Monographs, 142. Am. Math. Soc., Providence, RI (2007)

\bibitem[Ya1]{Ya1}
O.~S.~Yakimova, \emph{Weakly symmetric Riemannian manifolds with a reductive isometry group} (Russian), Mat. Sb. \textbf{195}  (2004), 143--160.

\bibitem[Ya2]{Ya2}
O.~S.~Yakimova, \emph{Gelfand Pairs}, Bonner Math. Schriften 374 (2005).

\bibitem[Z]{Z}
W.~Ziller,
\emph{Weakly symmetric spaces}, in: Topics in Geometry in Memory of Joseph D'Atri, Birkh\"{a}user, Basel, 1996, 355--368.

\end{thebibliography}
\end{document}